\documentclass{amsart}
\usepackage{amssymb}
\usepackage{color}

\numberwithin{equation}{section}

\newtheorem{theorem}{Theorem}

\newtheorem{thm}{Theorem}[section]
\newtheorem{lem}[thm]{Lemma}
\newtheorem{cor}[thm]{Corollary}
\newtheorem{prop}[thm]{Proposition}

\theoremstyle{definition}
\newtheorem{defn}[thm]{Definition}

\theoremstyle{remark}
\newtheorem*{rem}{Remark}

\newcommand{\R}{\mathbb{R}}

\newcommand{\dga}{\dot\gamma}

\newcommand{\be}{\begin{equation}}
\newcommand{\ee}{\end{equation}}

\begin{document}

\title{Rigidity of Busemann convex Finsler metrics}

\author{Sergei Ivanov}
\address{Sergei Ivanov:
St.\ Petersburg Department of Steklov Mathematical Institute,
Russian Academy of Sciences,
Fontanka 27, St.Petersburg 191023, Russia}
\email{svivanov@pdmi.ras.ru}

\author{Alexander Lytchak}
\address{Alexander Lytchak:
Mathematisches Institut,
Universit\"at K\"oln,
Weyertal 86--90, 50931, K\"oln, Germany}
\email{alytchak@math.uni-koeln.de}



\begin{abstract}
We prove that a Finsler metric is nonpositively curved in the sense of Busemann
if and only if it is affinely equivalent to a Riemannian metric of nonpositive sectional curvature.
In other terms, such Finsler metrics are precisely Berwald metrics
of nonpositive flag curvature.
In particular in dimension~2 every such metric is Riemannian or
locally isometric to that of a normed plane.
In the course of the proof we obtain new characterizations of Berwald metrics
in terms of the so-called linear parallel transport.
\end{abstract}

\maketitle

\section{Introduction}

The notion of nonpositive curvature in Riemannian geometry has two famous generalizations to metric geometry
due to Alexandrov and Busemann, respectively.
Alexandrov's generalization is nowadays known as locally CAT(0) spaces.
We refer to \cite{Ballmann}, \cite{BH} and the bibliography  therein for a vast literature on the subject.
Busemann nonpositively curved spaces, also known as {\it locally convex} spaces,
are a larger class of metric spaces defined as follows (see e.g.\ \cite{Papa}):

\begin{defn} \label{def:NPC}
A geodesic metric space $(X,d)$ is {\it Busemann convex} if
for every pair of constant-speed geodesics
$\gamma _{1,2} :[a,b]\to X$ the function $t\mapsto d(\gamma _1(t),\gamma _2(t))$
is convex on~$[a,b]$.

A metric space $(X,d)$ is {\it nonpositively curved in the sense of Busemann}
({\it Busemann NPC} for short)
if every point in $X$ has a Busemann convex neighborhood.
\end{defn}

In more geometric terms, a geodesic metric space $(X,d)$ is Busemann convex
if and only if for every geodesic triangle $\triangle abc$ in $X$, the distance between the
midpoints of its sides $[ab]$ and $[ac]$ is no greater than $\frac12 d(b,c)$.

Contrary to Alexandrov's definition of nonpositive curvature,
Busemann's one is satisfied by all normed vector spaces with strictly convex norms.
Thus it can be sensibly applied to Finsler metrics.
In fact, Finsler metrics are one of the main motivations in Busemann's work \cite{busemann48}
where the definition is introduced.
It is natural to ask how this class of Finsler metrics can be characterized
in differential geometric terms.
For discussions of this question, see \cite[\S13]{busemann48}, \cite{Kelly},
the introduction in \cite{KVK}, \cite{KristalyR} and Problem 35 in Z.~Shen's problem list \cite{Shenproblem}.

In this paper we settle this question.
It turns out that the Busemann NPC condition for (smooth and strictly convex) Finsler metrics
has surprising rigidity implications and very few metrics satisfy it.

For Riemannian manifolds the Busemann NPC condition
is equivalent to nonpositive sectional curvature.
Hence one can construct an open set of Riemannian
examples by perturbing any negatively curved metric.
Indeed, if a perturbed metric tensor is sufficiently $C^2$-close to the original one,
then it is also negatively curved and hence Busemann NPC.
One might expect that in the Finsler case
a similar property holds and that one can construct many examples of Busemann NPC Finsler metrics
by varying negatively curved Riemannian ones.
These expectations turn out to be wrong as our main theorem shows:

\begin{theorem} \label{mainthm}
A Finsler manifold $(M,F)$ is Busemann NPC if and only if there exists
a nonpositively curved Riemannian metric $g$ on $M$  whose
Levi-Civita connection preserves the Finsler norm~$F$.
\end{theorem}

The Finsler norm $F$ in Theorem \ref{mainthm} is not assumed to be reversible.
Strictly speaking, non-reversible Finsler manifolds are not metric spaces
since the distance lacks symmetry.
Nevertheless Definition \ref{def:NPC} applies just as well,
cf.\ \cite[Section~5]{KVK}.

In the language of Finsler geometry our result reads as follows:

\begin{theorem} \label{thm:conj}
A Finsler manifold $(M,F)$ is Busemann NPC if and only if
it is  a Berwald manifold of nonpositive flag curvature.
\end{theorem}

The definitions of Berwald metrics and flag curvature
are given in Section~\ref{sec:prel}.
Theorem \ref{thm:conj} is essentially a reformulation of Theorem \ref{mainthm},
their equivalence easily follows from Szab\'o's metrization theorem \cite{Szabo}.
The ``if''  direction of Theorem~\ref{thm:conj} is proved by
Krist\'aly,  Varga and Kozma \cite{KVK}, see also~\cite{KK}.
Results of \cite{KK} also imply that
if $(M,F)$ is Berwald and Busemann NPC, then it has nonpositive flag curvature.

\medskip

We now discuss some implications of Theorems 1 and~2
and well-known features of Berwald metrics.
First, Theorem \ref{mainthm} implies that in a connected Busemann NPC Finsler manifold
all tangent spaces are isometric as normed vector spaces.
This is already a strong restriction on a Finsler metric.

Furthermore, if $(M,F)$ and $g$ are as in Theorem \ref{mainthm}
then the Finsler norm at every point
is invariant under the holonomy group of~$(M,g)$.
Hence, if the holonomy group acts transitively on the unit sphere
then the Finsler structure must be Riemannian.
In dimension~2 the local holonomy group is either transitive or trivial,
thus we have the following (cf.~\cite[Corollary 4.3.5]{Chern}):

\begin{cor}
If a 2-dimensional connected Finsler manifold is Busemann NPC,
then it is Riemannian or locally isometric to a normed plane.
\end{cor}

In general, recall that the Riemannian holonomy group is transitive
unless the metric locally splits as a product or is locally symmetric \cite{Berger}.
This leads to a classification of Busemann NPC Finsler metrics
exactly as that of Berwald metrics
(see \cite{Szabo} or \cite[Theorem 4.3.4]{Chern})
with additional requirements that
the symmetric spaces involved are of non-compact type
and the Riemannian factors are nonpositively curved.

A Finsler norm $F$ is preserved by the Levi-Civita parallel transport
of a Riemannian metric $g$ if and only if $F$ and $g$ are affinely equivalent,
that is if they have the same geodesics up to affine reparametrizations
(see e.g.\ \cite[Theorem 4.1.3]{Chern}).
Thus we have the following short reformulation of Theorem~\ref{mainthm}:

\begin{cor} \label{cor:affine}
A Finsler metric is Busemann NPC if and only if it is affinely equivalent
to a nonpositively curved Riemannian metric.
\end{cor}

We mention that Finsler manifolds affinely equivalent to nonpositively curved Riemannian symmetric spaces
have recently turned up  and played a prominent role in a series of papers
of B.~Leeb, M.~Kapovich and J.~Porti, see \cite{KLP} and references therein.

\subsection*{Characterizations of Berwald spaces}
Due to known results mentioned above,
in order to prove Theorem~\ref{mainthm} it   would suffice to show
that every Busemann NPC Finsler space is Berwald.
However, we prove the ``only if'' part of Theorem~\ref{mainthm} directly,
bypassing Szabo's metrization theorem.
The proof requires very little background from Finsler geometry,
summarized in Section \ref{sec:prel}.

The proof consists of two steps. In the first step, contained in Section \ref{sec:bus},
we observe that in a Busemann NPC Finsler space
the linear parallel transport
along any geodesic preserves the Finsler norm.
This follows from analysis of Jacobi fields near points where they vanish.

The above mentioned \emph{linear parallel transport} is defined in Section \ref{sec:prel}.
Note that this notion is different from the more commonly used ``canonical parallel transport'',
which is usually non-linear.
See e.g.\ \cite[Chapter~4]{Chern} where the two parallel transports
are considered together.
The non-linear parallel transport is not used in this paper
beyond the comparison remarks in this introduction.

In the second step, contained in Proposition \ref{prop:1} (see Section \ref{sec:ber}),
we prove the following general fact about Finsler metrics:
\emph{If the linear parallel transport of a Finsler manifold $(M,F)$ preserves $F$,
then $F$ is Berwald.}
In the proof (see also Proposition \ref{prop:fins})
we construct a Riemannian metric affinely equivalent to $F$.
This Riemannian metric is then used
in the proof of Theorem~\ref{mainthm} in Section \ref{sec:proof}.

\begin{rem}
We suggest the reader to compare Proposition \ref{prop:1} with well-known facts about
the canonical (non-linear) parallel transport:
It always preserves the Finsler norm,
and it is linear if and only if the metric is Berwald.
Proposition~\ref{prop:1} ``mirrors'' the last mentioned characterization
with the linear parallel transport,
which is by definition linear but does not, in general, preserve the norm.
Note that for Berwald metrics the two parallel transports coincide,
see \cite[\S4.3]{Chern}.
\end{rem}

As a by-product, we obtain  another characterization of Berwald spaces
in terms of the \emph{linear  holonomy group}
naturally defined via the linear parallel transport.
Namely in Proposition \ref{prop:2} we prove the following:
\emph{The closure of the linear holonomy group of a Finsler metric
is compact if and only if the metric is Berwald.}

We mention that similar questions for the non-linear holonomy group
(defined via the non-linear parallel transport)
were recently studied in \cite{Muzsnay}.

\subsection*{Beyond smoothness}
We work only with smooth Finsler structures and
the smoothness is essential in our proofs.
It would be interesting to extend some of the results
to non-smooth Finsler structures and more general metric spaces.
In particular Corollary \ref{cor:affine} suggests the following  questions.

\emph{Question 1: Given a Busemann convex metric space, is there a CAT(0) space closely and naturally  related to it?}

\emph{Question 2:  Let $(X,d)$ be a geodesic space affinely equivalent to a Busemann convex space $(X,d_0)$.
Is it true that $(X,d)$ must be Busemann convex as well?}

In the slightly more general class of spaces with convex bicombings the first question turns out to be of interest in relation with the theory of Gromov hyperbolic groups,
see \cite{Langhyp}.  Some examples showing that one cannot expect affine equivalence, as in Corollary \ref{cor:affine}, are discussed in
  Section~\ref{sec:bicombing}.

 Metric  spaces affinely equivalent to Riemannian manifolds
are characterized in \cite{Lyt}
in a way similar to Szab\'o's metrization theorem.
In fact, every such space is a limit of smooth Finsler metrics
affinely equivalent to the same Riemannian one,
see Theorem~1.4 and Lemma~4.1 in \cite{Lyt}.
This and Corollary \ref{cor:affine} imply an affirmative answer to  the second question above
 if $(X,d_0)$ is a  smooth Finsler manifold.

\subsection*{Acknowledgements}
The authors thank David Bao, Martin Kell, Alexandru Kristaly,  Hans-Bert Rademacher, Jozsef Szilasi for helpful comments.
Special thanks are due to Urs Lang for a discussion  which led to the examples
in Section~\ref{sec:bicombing}.
The first author was supported in part by the RFBR grant 17-01-00128-a
and  by the Presidium of the Russian Academy of Sciences grant PRAS-18-01.
The second author was supported in part by the DFG grants SFB TRR 191 and SPP 2026.

\section{Preliminaries} \label{sec:prel}

In this section we recall some basics in Finsler geometry
and prove some auxiliary facts.
We follow the presentation in \cite{Rad},
where most concepts are developed from Riemannian point of view.
We refer to \cite{Chern} and \cite{BCS} for more exhaustive exposition of Finsler geometry.
We include proofs of some standard facts to keep the paper accessible
to readers not familiar with Finsler geometry.

Let $(M,F)$ be a Finsler manifold, with the usual smoothness and strict convexity assumptions
on the Finsler norm $F:TM\to\R_+$.
These assumptions ensure that $F$ determines a smooth geodesic flow on $TM\setminus\{0\}$.
As mentioned in the introduction we do not assume that $F$ is reversible.
We denote by $d_F$ the (non-symmetric) distance function induced by $F$ on~$M$.

By a {\it geodesic} we always mean an affine geodesic, i.e., a constant-speed one.
Constant paths are not regarded as geodesics.
Since all our considerations are local, we may always assume that the manifold $M$
is narrowed down to a small open region where all geodesics are embedded.

For a nonzero $v\in TM$, we denote by $\gamma _v$ the unique
geodesic with initial data $\dga_v(0)=v$.
Note that $\gamma_v(t)$ depends smoothly on $v$ and~$t$.

A {\it Jacobi field} along a geodesic $\gamma$ is a variation field
of a smooth family of geodesics.
We denote by $\mathcal J^F(\gamma )$ the set of all Jacobi fields along $\gamma$.

We make use of several notions from Finsler geometry.
They all can be defined by means of osculating Riemannian metrics,
see below.

\subsection{Osculating Riemannian metrics}

For every $p\in M$ and $v\in T_pM\setminus\{0\}$ there is
a unique positive definite quadratic form $g_v$ on $T_pM$
such that $g_v$ and $F^2|_{T_pM}$ agree to second order at~$v$.
If $V$ is a non-vanishing vector field on an open set $U\subset M$,
the family $\{g_{V(p)}\}_{p\in U}$ of quadratic forms
defines a Riemannian metric $g_V$ on~$U$.
If $\gamma$ is an embedded geodesic and $V$ extends
the velocity field of $\gamma$ to a neighborhood of $\gamma$,
we call $g_V$ an \textit{osculating Riemannian metric} for $\gamma$
and denote it by $g^\gamma$.
Note that $g^\gamma$ is uniquely determined at every point on $\gamma$
but the extension to a neighborhood is not unique.

We need the following property
(see Lemma 4.4 and Proposition 5.1 in \cite{Rad} or Chapter~8 in \cite{Sh01}):

\begin{lem} \label{lem:osculating}
Let $\gamma$ be an embedded geodesic of $(M,F)$ and
$g=g^\gamma$ an osculating Riemannian metric for $\gamma$.
Then $\gamma$ is a geodesic of $g$ with the same
space of Jacobi fields:
$\mathcal J^g(\gamma ) = \mathcal J^F(\gamma)$.
\end{lem}

We comment that Lemma \ref{lem:osculating} is essentially a consequence
of the fact that respective equations
from calculus of variations involve only first and second derivatives
of a Lagrangian with respect to the velocity variable.

\subsection{Linear parallel transport and Jacobi operator}
\label{sub:par}

An osculating Riemannian metric $g^\gamma$ for a geodesic $\gamma$
defines the following structures,
which are in fact independent on the choice of $g^\gamma$
(see \cite{Rad} or Lemma \ref{lem:prel} below):
\begin{itemize}
 \item The covariant derivative of a vector field $W$ along $\gamma$,
denoted by $D_{\gamma}W$:
\be \label{eq:covdev}
  D_{\gamma}W(t) =\frac{\nabla^\gamma}{dt} W(t)
\ee
where $\nabla^\gamma$ is the Levi-Civita connection of $g^\gamma$.
 \item The linear parallel transport along $\gamma$ associated to this connection.
It is a family of non-degenerate linear maps between tangent spaces $T_{\gamma(t)}M$.
 \item The Jacobi operator $R^\gamma$, which is a family of linear operators
on the tangent spaces $T_{\gamma(t)}M$ satisfying the Jacobi equation
\be\label{eq:jacobi}
 D_{\gamma}D_{\gamma}J(t) = -R^\gamma(J(t))
\ee
for all Jacobi fields $J$ along $\gamma$.
Namely $R^\gamma(w)=R_{g^\gamma}(w,\dga(t))\dga(t)$
for all $w\in T_{\gamma(t)} M$, where $R_{g^\gamma}$
is the Riemannian curvature tensor of $g^\gamma$.
\end{itemize}
Our notation has $\gamma$ rather than $\dga$ in indices (cf.~\cite{Rad})
since we are only interested in vector fields along geodesics.

The following simple computation is used several times in the paper.

\begin{lem}\label{lem:jw}
Let $\gamma=\gamma_v$ be a geodesic of $(M,F)$
and $w\in T_{\gamma(0)}M$. Define a Jacobi field $J$ along $\gamma$ by
\be \label{eq:jw-def}
 J(t)= \frac d {ds} \Big  |_{s=0} \gamma _{v+sw}  (t) \,,
\ee
and let $W$ be a $D_{\gamma}$-parallel vector field along $\gamma$
(with respect to any osculating Riemannian metric) with $W(0)=w$.
Then
\be \label{eq:jw}
 J(t) = t\cdot W(t) + O(t^3) \, , \qquad t\to 0.
\ee
\end{lem}

\begin{proof}
Fix an osculating Riemannian metric $g^\gamma$ for \eqref{eq:covdev} and \eqref{eq:jacobi}.
For brevity, in this proof we use notation $X'$ for
the covariant derivative $D_{\gamma}X$ of a vector field $X$ along $\gamma$.
From \eqref{eq:jw-def} we have $J(0)=0$,
$$
 J'(0)= \frac{\nabla^\gamma}{dt}\Big|_{t=0} \frac {d}{ds}\Big|_{s=0} \gamma_{v+sw}(t)
 = \frac{d}{ds}\Big|_{s=0} \frac {d}{dt}\Big|_{t=0} \gamma_{v+sw}(t) = w \,,
$$
and
$
 J''(0) = - R^\gamma(0) = 0
$.
The vector field $X(t):=t\cdot W(t)$ has the same value
and the same first and second covariant derivatives at $t=0$, namely
$X(0)=0$, $X'(0)=W(0)=w$ and $X''(0)=2W'(0)=0$
since $W$ is $D_{\gamma}$-parallel.
Therefore $J(t)-X(t)=O(t^3)$ as $t\to 0$.
\end{proof}

\begin{lem} \label{lem:prel}
Let $g_1,g_2$ be two Riemannian metrics on $M$.
Suppose that a path $\gamma$ is a geodesic for both $g_1$ and $g_2$
and the two metrics induce the same Jacobi fields along $\gamma$,
i.e.\ $\mathcal J^{g_1}(\gamma)=\mathcal J^{g_2}(\gamma)$.
Then $g_1$ and $g_2$ induce the same covariant derivative
and the same Jacobi operator along $\gamma$.

In particular, for a Finsler manifold $(M,F)$ and an $F$-geodesic $\gamma$,
the operators $D_{\gamma}$ and $R^\gamma$ are independent of
the choice of the osculating metric $g^\gamma$.
\end{lem}

\begin{proof}
We apply the computation from Lemma \ref{lem:jw} to the Riemannian metrics $g_i$.
The formula \eqref{eq:jw} determines the first
order jet of $W(t)$ at $t=0$ in terms of the Jacobi field $J$.
Hence the space $\mathcal J=\mathcal J^{g_i}(\gamma)$
determines (independently of the metric)
the set of vector fields $W$ along $\gamma$ such that $D_{\gamma}W(0)=0$.
Indeed, a vector field satisfies $D_{\gamma}W(0)=0$ if and only if
there exists $J\in\mathcal J$ such that \eqref{eq:jw} holds.
Similarly, for every $t_0\in\R$ the set of vector fields $W$ along $\gamma$
satisfying $D_{\gamma}W(t_0)=0$ is determined by $\mathcal J$.

Thus the set of parallel fields along $\gamma$ is determined by $\mathcal J$
and hence it is the same set for $g_1$ and~$g_2$.
The parallel fields determine the covariant derivative $D_{\gamma}$.
The Jacobi equation \eqref{eq:jacobi}, $\mathcal J$ and $D_{\gamma}$
together determine the Jacobi operator $R^\gamma$.
This proves the first statement of the lemma.
The second one follows by Lemma~\ref{lem:osculating}.
\end{proof}

With Lemma \ref{lem:prel} at hand,
one can define the covariant derivative $D_\gamma$
and the Jacobi operator $R^\gamma$ along
a Finsler geodesic~$\gamma$ as those of any osculating Riemannian metric
associated to~$\gamma$.
Note that with these definitions the first part of
Lemma \ref{lem:prel} applies to Finsler metrics as well.
In particular, affinely equivalent Finsler metrics
induce the same covariant derivatives and the same Jacobi operators.

Following \cite{Chern}, we use the term {\em linear parallel transport}
for the parallel transport operator induced by~$D_\gamma$
and refer to $D_{\gamma}$-parallel vectors fields
as \textit{linearly parallel} vector fields along $\gamma$.

\subsection{Flag curvature}
\label{sub:flag}

The {\it flag curvature} $K_F(v,\sigma)$ of a Finsler manifold $(M,F)$ is a function of
a point $p\in M$, a plane (i.e., two-dimensional linear subspace)
$\sigma\subset T_pM$, and a direction $v\in\sigma\setminus\{0\}$
in that plane.
It can be defined as the sectional curvature $K_g(\sigma)$
of an osculating Riemannian metric $g=g^\gamma$ for the geodesic $\gamma=\gamma_v$.
From the Riemannian Jacobi equation one sees that
\be\label{eq:curv}
  K_F(v,\sigma) = K_g(\sigma)=\frac{g(R^\gamma(w),w)}{|v\wedge w|^2_g}
\ee
for any vector $w\in\sigma$ such that $v$ and $w$ are linearly independent,
where
$$
 |v\wedge w|^2_g=g(v,v)g(w,w)-g(v,w)^2 .
$$
The formula \eqref{eq:curv} and Lemma \ref{lem:prel} imply that the definition
does not depend on the choice of the osculating metric $g$.
If the metric $F$ is Riemannian, then $K_F(v,\sigma)$
equals its sectional curvature at~$\sigma$,
in particular it does not depend on $v$.

Consider the Jacobi operator $R^\gamma$ at $p=\gamma(0)$.
By symmetries of the Riemannian curvature tensor, this operator
on $T_pM$ is symmetric with respect to the inner product induced by $g=g^\gamma$.
Hence it has only real eigenvalues.
By \eqref{eq:curv}, the flag curvature of $(M,F)$ is nonpositive
if and only if these eigenvalues are all nonpositive
for every $p$ and $\gamma$.
Since Jacobi operators of affinely equivalent metrics coincide,
this implies the following:

\begin{lem} \label{lem:nonpos}
Let $F_1$ and $F_2$ be affinely equivalent Finsler metrics
and the flag curvature of $F_1$ is nonpositive.
Then the flag curvature of $F_2$ is nonpositive as well.
\end{lem}

\subsection{Berwald metrics}
\label{sub:berwald}

Berwald metrics are a special class of Finsler metrics that can be
defined in many equivalent ways, see \cite[Ch.~10]{BCS} and \cite[Ch.~4]{Chern}.
A definition convenient for our purposes is the following.
A Finsler manifold $(M,F)$ is Berwald if and only if there exists
a symmetric affine connection on $M$ such that the geodesics of $F$
are also geodesics of this connection.

By Szab\'o's metrization Theorem \cite{Szabo}, this affine connection
can be realized as the Levi-Civita connection of a Riemannian metric.
Thus a Finsler metric is Berwald if and only if it is affinely equivalent
to a Riemannian metric.

Finally, we need the fact that if the Levi-Civita connection of
a Riemannian metric $g$ preserves a Finsler norm $F$,
then $F$ and $g$ are affinely equivalent
(and hence $F$ is Berwald).
See e.g.\ \cite[Lemma 4.1.5]{Chern} for a proof.
Alternatively, one can verify directly that in this case
the Euler-Lagrange equation for Finsler geodesics is satisfied by Riemannian geodesics.
This is easy to see in the Riemannian exponential coordinates
at the point in question.

\section{An implication of Busemann convexity} \label{sec:bus}

The aim of this section is the following

\begin{lem}  \label{lem:first}
Let $(M,F)$ be a Busemann NPC Finsler manifold
and $\gamma:(a,b)\to M$ a geodesic.
Then the linear parallel  transport  along $\gamma$ preserves  $F$,
i.e., for every linearly parallel vector field $W$ along~$\gamma$ the function
$t\mapsto F(W(t))$ is constant.
\end{lem}

\begin{proof}
The statement is local, thus we may assume that $\gamma$ is contained
in a Busemann convex open subset of~$M$.

First observe that for every Jacobi field $J$ along $\gamma$,
its norm $F(J(t))$ is a convex function of~$t$.
Indeed, let $\{\gamma_s\}$ be a geodesic variation of $\gamma$
whose variation field is~$J$.
By Busemann convexity, for every $s$ the function
$
 f _s(t)= d_F(\gamma (t), \gamma _s(t))
$
is convex.
Hence so is the limit function
$ \lim_{s\to 0} \frac{f_s(t)}{s}  = F(J(t)) $.

Now let $W\ne 0$ be a linearly parallel vector field along $\gamma$.
Note that $F(W(t))$ is a smooth function of~$t$.
Assuming that $0\in (a,b)$, Lemma \ref{lem:prel} implies that
there is a Jacobi field $J$ along $\gamma$ such that
$$
 J(t)= t\cdot  W(t) + O (t^3)\,, \qquad t\to 0 \,.
$$
For this Jacobi field $J$ and $t>0$ we have
\begin{equation} \label{eq:geq}
F(J(t)) = F(t\cdot W(t)) + O (t^3)
 =t \cdot F(W(t)) +O(t^3) \,, \qquad t\to 0_+.
\end{equation}
In particular $F(J(t))=t\cdot F(W(0)) + o(t)$ as $t\to 0_+$.
Since $F(J(t))$ is a convex function,
it follows that $F(J(t)) \ge t\cdot F(W(0))$ for all $t\ge 0$.
This and \eqref{eq:geq} imply that
$$
  F(W(t)) \ge F(W(0)) - O(t^2)  \,, \qquad t\to 0_+ \,,
$$
hence $F(W(t))'_{t=0} \ge 0$.

Similarly, for the Jacobi field $-J$ and $t<0$ we have
$$
F(-J(t)) = -t \cdot F(W(t)) +O(t^3) \,, \qquad t\to 0_- \,,
$$
and therefore $F(-J(t))\ge -t \cdot F(W(0))$ for all $t\le 0$.
These relations imply that
$$
  F(W(t)) \ge F(W(0)) - O(t^2)  \,, \qquad t\to 0_- \,,
$$
hence $F(W(t))'_{t=0} \le 0$.
Thus  $F(W(t))'_{t=0}=0$.

Shifting the parameter of $\gamma$ we deduce that
the derivative of $F(W(t))$ vanishes everywhere on $(a,b)$.
Therefore $F(W(t))$ is constant.
\end{proof}

\section{Linear parallel transport and Berwaldness} \label{sec:ber}

In this section we prove three characterizations of Berwald metrics
in terms of the linear parallel transport.
The proof of Theorems \ref{mainthm} and \ref{thm:conj}
in the next section uses Proposition \ref{prop:1},
which in its turn refers to Proposition \ref{prop:fins}.

\begin{prop} \label{prop:fins}
Let $(M,F)$ be a Finsler manifold. Let $g$ be a smooth Riemannian metric on $M$ such
that the linear parallel transport (determined by~$F$)
along any $F$-geodesic preserves~$g$.
Then  $g$ is affinely equivalent to~$F$,
in particular $F$ is Berwald.
\end{prop}

\begin{proof}
Let $\nabla$ denote the Levi-Civita connection induced by~$g$.
For $p\in M$ and $v\in T_pM\setminus\{0\}$ define
 $\kappa(v)\in T_pM$ by
$$
 \kappa(v) = \frac{\nabla}{dt} \Big|_{t=0} \dga_v(t)
$$
where $\gamma_v$ is, as usual, the $F$-geodesic with $\dga_v(0)=v$.
That is, $\kappa(v)$ is the Riemannian second derivative of~$\gamma_v$.
Note that $\kappa(v)$ depends smoothly on~$v$.  We need to show that
every geodesic of the Finsler manifold $(M,F)$ is a geodesic in $(M,g)$. Thus,
it suffices to prove that $\kappa(v)=0$ for all $p$ and~$v$.

Let $W$ be a linearly parallel vector field
along $\gamma_v$ and $w=W(0)$.
We claim that
\begin{equation} \label{eq:der}
 \frac d {ds}  \Big | _{s=0} \kappa(v+sw) = 2\cdot \frac {\nabla} {dt} \Big |_{t=0} W(t)\; .
\end{equation}
To prove this,
we use the Riemannian exponential coordinates with respect to~$g$ at~$p$
to identify a small neighborhood $U$ of $p$ with an open subset of $\R^n=T_pM$.
By means of coordinates we also identify vector fields along $\gamma_v$
with $\R^n$-valued functions of~$t$.
Recall that the Levi-Civita connection coefficients in exponential coordinates
vanish at the origin.
Hence in these coordinates we have
\be\label{eq:kappa-coord}
 \kappa(v) = \frac {d^2}{dt^2} \Big |_{t=0} \gamma_v(t)
\ee
and
$$
 \frac{\nabla}{dt} \Big |_{t=0} W(t) = \frac{d}{dt} \Big |_{t=0} W(t) .
$$
Differentiating \eqref{eq:kappa-coord} yields
$$
 \frac d {ds}  \Big | _{s=0} \kappa(v+sw)
 = \frac {\partial^3} {\partial s\partial t^2}  \Big | _{(s,t)=(0,0)} \gamma _{v+sw} (t)
 = \frac {d^2} {dt^2}  \Big  |_{t=0}  J(t)
$$
where $J$ is the Jacobi field along $\gamma_v$ defined by \eqref{eq:jw-def},
see Lemma \ref{lem:jw}.
By the relation $J(t)= t\cdot W(t) + O(t^3)$ from Lemma \ref{lem:jw},
the right-hand side can be rewritten as
$$
 \frac {d^2} {dt^2}  \Big  |_{t=0}  J(t)
 =\frac {d^2} {dt^2}  \Big  |_{t=0} (t\cdot W(t))
 = 2 \cdot \frac{d}{dt} \Big |_{t=0} W(t)
 = 2\cdot \frac{\nabla}{dt} \Big |_{t=0} W(t) \,.
$$
This finishes the proof of \eqref{eq:der}.

\medskip

Now we forget about coordinates and introduce some shortcut notation.
We write $\langle \cdot,\cdot\rangle$ instead of $g(\cdot,\cdot)$
and denote by $X'(t)$ the Levi-Civita derivative (induced by~$g$)
of a vector field $X(t)$ along a path.

Since $g$ is preserved by the linear parallel transport,
for any linearly parallel vector fields $X$ and $Y$ along $\gamma _v$,
the inner product $\langle X(t),Y(t)\rangle$ is constant, hence
\begin{equation} \label{eq:scalar}
  0 = \frac d{dt} \langle X(t),Y(t)\rangle
  =\langle X',Y\rangle + \langle X,Y' \rangle \,.
\end{equation}
Note that $\dga_v$ is linearly parallel along $\gamma_v$.
For  $X=Y=\dga_v$ and $t=0$ the identity \eqref{eq:scalar} boils down to
\begin{equation} \label{eq:v}
 \langle \kappa(v), v \rangle = 0 .
\end{equation}
Fix $w\in T_pM$ and let $W$ be the linearly parallel vector field along $\gamma _v$ with $W(0)=w$.
We replace  $v$ by $v+sw$ in \eqref{eq:v} and differentiate with respect to~$s$
using \eqref{eq:der}:
$$
  0=\frac d{ds}  \Big |_{s=0} \langle \kappa(v+sw), v+sw \rangle
  = \langle 2\cdot W'(0), v \rangle +  \langle \kappa(v), w\rangle \,.
$$
On the other hand, plugging  $X=W$ and $Y=\dga_v$ into \eqref{eq:scalar}
we see that
$$
 0 = \langle W'(0), v \rangle + \langle w, \kappa(v) \rangle \, .
$$
The last two equations together imply that $\langle \kappa(v), w\rangle =0$.
Since $w$ was arbitrary in $T_pM$, we conclude that $\kappa(v)=0$.
This finishes the proof of the proposition.
\end{proof}

We now deduce the main ingredient of the proofs of Theorems \ref{mainthm} and \ref{thm:conj}.

\begin{prop} \label{prop:1}
Let $(M,F)$ be a Finsler manifold.
Suppose that the linear parallel transport along any geodesic preserves
the Finsler norm $F$.
Then there exists a Riemannian metric on $M$ affinely equivalent to $F$.
In particular $F$ is Berwald.
\end{prop}

\begin{proof}
Consider a smooth Riemannian metric $g$ on $M$
whose value $g_p$ at every $p\in M$
is canonically determined by the norm $F_p:=F|_{T_pM}$,
where ``canonically'' means that
any linear isometry between normed spaces
$(T_pM,F_p)$ and $(T_qM,F_q)$ takes $g_p$ to~$g_q$.
See \cite{Matveev} for a possible construction of such~$g$.

Since the linear parallel transports are linear and preserve $F$,
they also preserve~$g$.  The claim now follows from
Proposition \ref{prop:fins}.
\end{proof}

Our next result generalizes Proposition \ref{prop:1}.
In order to state it we need the notion of linear holonomy group,
defined as follows.
Let $(M,F)$ be a Finsler manifold.
For a piecewise geodesic path $\gamma:[a,b]\to M$ denote by $P_\gamma$ the
linear parallel transport along $\gamma$.
Recall that $P_\gamma$ is a linear isomorphism from $T_{\gamma(a)}M$ to $T_{\gamma(b)}M$.
Fix a point $p\in M$ and let $GL(T_pM)$ denote the group of
linear self-isomorphisms of $T_pM$.
The \emph{linear holonomy group} $LH_p(M,F)$ is the subgroup of $GL(T_pM)$
generated by maps of the form $P_{\gamma_2}^{-1}\circ P_{\gamma_1}$
where $\gamma_1$ and $\gamma_2$ are piecewise geodesic paths
starting at $p$ and having a common end point.
(The definition is so cumbersome because linear parallel transports
in opposite directions are in general not inverse to each other.)

\begin{prop} \label{prop:2}
Let $(M,F)$ be a connected Finsler manifold and $p\in M$.
Then the closure of the linear holonomy group $LH_p(M,F)$ in $GL(T_pM)$
is compact if and only if $F$ is Berwald.
\end{prop}

\begin{proof}
Let $H=LH_p(M,F)$.
If $F$ is Berwald then by Szab\'o's metrization theorem
the linear parallel transport is the Levi-Civita
parallel transport of some Riemannian metric.
Thus $H$ preserves an inner
product on $T_pM$ and is therefore contained
in the corresponding orthogonal group, which is compact.

On the other hand, if the closure of $H$ is compact then
there exists an $H$-invariant inner product $g_p$ on $T_pM$.
For every $x\in M$, pick a piecewise geodesic $\gamma$ connecting $p$ to~$x$
and define an inner product $g_x$ on $T_xM$ as the push-forward of $g_p$ by
the linear parallel transport $P_\gamma$.
Since $g_p$ is $H$-invariant, the resulting Riemannian metric $g=\{g_x\}_{x\in M}$
is well-defined and it is preserved by the linear parallel transport along any $F$-geodesic.
It is easy to see that $g$ is smooth.
By Proposition \ref{prop:fins} it follows  that $F$ is a Berwald metric.
\end{proof}

\section{Proof of the theorems}\label{sec:proof}

We can now collect the harvest and prove
Theorems \ref{mainthm} and \ref{thm:conj}.

The ``if'' part of Theorem \ref{thm:conj} follows from the main result of \cite{KVK}.
To prove the ``if'' part of Theorem \ref{mainthm},
let $g$ be a Riemannian metric on $M$ of nonpositive sectional curvature
whose Levi-Civita connection preserves the Finsler norm~$F$.
Then $F$ is Berwald and affinely equivalent to~$g$
(see Section \ref{sub:berwald}).
Now Lemma \ref{lem:nonpos} implies that $F$
has nonpositive flag curvature.
Hence, by the ``if'' part of Theorem~\ref{thm:conj},
$(M,F)$ is Busemann NPC.

Now we prove the ``only if'' implications of the theorems.
Let $(M,F)$ be a Busemann NPC Finsler manifold.
Then by Lemma \ref{lem:first} the linear parallel
transport preserves~$F$.
Therefore, by Proposition \ref{prop:1} the Finsler metric $F$ is
Berwald and affinely equivalent to a Riemannian metric $g$.

It remains to prove that $g$ has nonpositive sectional curvature
and $F$ has nonpositive flag curvature.
This can be seen from results of \cite{KK}
but for reader's convenience we give a short proof here.

First let us show that $F$ has nonpositive flag curvature.
Suppose the contrary.
Then, as explained in Section~\ref{sub:flag},
for some geodesic $\gamma$ the Jacobi operator $R^\gamma$
has a positive eigenvalue at $p=\gamma(0)$.
Let $w\in T_pM$ be a nonzero vector such that
$R^\gamma(w)=\lambda w$, $\lambda>0$.
Consider the Jacobi field $J$ along $\gamma$ with the initial conditions
$J(0)=w$ and $D_\gamma J(0)=0$.
Let $\overline J(t)\in T_pM$ be the image of $J(t)$ in $T_pM$ under the
linear parallel transport along $\gamma_v$.
Then $\overline J'(0)=D_\gamma J(0)=0$ and
$$
 \overline J''(0)= D_\gamma D_\gamma J(0) = -R^\gamma(J(0))=-\lambda w \,.
$$
Since the linear parallel transport preserves $F$, we have $F(J(t))=F(\overline J(t))$ for all~$t$.
Therefore
$$
 F(J(t))''_{t=0} = F(\overline J(t))''_{t=0} = -\lambda\cdot F(w) < 0 \,,
$$
contrary to the fact that $F(J(t))$ is a convex function (see the proof of Lemma~\ref{lem:first}).
This contradiction shows that $F$ has nonpositive flag curvature
and finishes the proof of Theorem \ref{thm:conj}.

Now Lemma \ref{lem:nonpos} and the affine equivalence of $F$ and $g$ imply
that $g$ has nonpositive sectional curvature.
This finishes the proof of Theorem~\ref{mainthm}.

\section{Non-smooth examples}\label{sec:bicombing}
 In this section we describe two simple examples of Busemann convex spaces not affinely equivalent to any CAT(0) space.
Both examples are constructed from a plane $V$ with a strictly convex non-Euclidean norm.

The first example $X_1$ arises from $V$ by attaching a ray  at the origin.

The second example $X_2$ is the double branched cover of $V$ with branching locus the origin $0\in V$. We give $X_2$  the naturally induced length metric.
Note that $X_2$ is biLipschitz to  the Euclidean plane.

It is not difficult to see that $X_1$ and $X_2$ are Busemann convex. Assume now that there exists a space $Z$ and an affine  bijection $F:X_i  \to Z$
for $i=1$ or $i=2$, hence $F$ sends constant speed geodesics to constant speed geodesics.

Every ray in $V$ starting at the origin can be continued in $X_1$ to an infinite geodesic by the attached ray in $X$. Thus, in the case of $X_1$, any  ray in $V$ is stretched by the same
factor   as the attached ray. Hence all rays in $V$ through the origin are stretched by the same factor. The metric on the punctured plane $Z_0=F(V\setminus \{0 \}  )$
is affinely equivalent to the flat Riemannian manifold $\R^2\setminus \{ 0\}$.   Due to  \cite[Theorem 1.5]{Lyt}, the metric on $Z_0$ must come from a constant norm on $V$.
In particular,  parallel rays in $V$ must be stretched by $F$ by equal factors.
 We conclude that
 $F$ must stretch all distances by the same factor.
Thus, $F$ is a dilation and $Z$  cannot be CAT(0).

Every pair of rays $\gamma_{1,2}$ in $X_2$ starting at the origin  can be continued to infinite  geodesics by the same ray $\gamma _3$. Therefore, also in case of $X_2$,  the map  $F$ must stretch all rays through the origin by the same factor.  Using  \cite[Theorem 1.5]{Lyt} as above,
 this again implies that $F$ is just a dilation and $Z$ cannot be CAT(0).

However,  in both cases there is a natural CAT(0) metric on the spaces $X_1$, $X_2$ which arises in the same way as $X_{1,2}$  from the Euclidean plane $V_0$ instead of $V$.
These CAT(0) metrics have the same unparametrized geodesics as the original metrics  and are affinely equivalent to the original metric outside the branching locus.

\bibliographystyle{alpha}
\bibliography{busemnpc}

\end{document}